\documentclass[11pt]{amsart}

\usepackage[margin=1.0in]{geometry}
\usepackage{amsmath}
\usepackage{amssymb}
\usepackage{amsthm}
\usepackage{enumerate}
\usepackage{xcolor}
\usepackage{tikz}
\usepackage[indent]{parskip}
\usepackage{hyperref}
\newtheorem{thm}{Theorem}[section]
\newtheorem{lem}[thm]{Lemma}
\newtheorem{cor}[thm]{Corollary}

\theoremstyle{definition}
\newtheorem{dfn}[thm]{Definition}

\newtheorem*{prb*}{Problem}

\theoremstyle{remark}

\providecommand{\al}{\alpha}

\providecommand{\emp}{\varnothing}

\providecommand{\R}{\mathbb{R}}

\providecommand{\F}{\mathbb{F}}

\providecommand{\N}{\mathbb{N}}

\providecommand{\cA}{\mathcal{A}}

\providecommand{\cH}{\mathcal{H}}

\providecommand{\cO}{\mathcal{O}}

\providecommand{\gdot}{G_t^{\textup{dot}}}

\providecommand{\hdot}{G_t^{\textup{dot}}}
\providecommand{\hdist}{G_t}

\renewcommand{\phi}{\varphi}

\providecommand{\sep}{\hspace{.25in}\text{and}\hspace{.25in}}

\newcommand{\neigh}{\widetilde{N}}

\numberwithin{equation}{section}

\title{VC-dimension of distance graphs in finite field geometries}

\author{Brian McDonald Bakas}
\address{Indianapolis, IN}
\email{brian.mcdonald.bakas@gmail.com}

\author{Moustapha Diallo}
\address{Atlanta, GA}
\email{moustapha.diallo@uga.edu}

\author{Alex McDonald}
\address{Mathematics Department, Kennesaw State University, Marietta, GA}
\email{amcdon79@kennesaw.edu}

\thanks{The third listed author is supported by an AMS-Simons
Research Enhancement Grant for Primarily Undergraduate Institution Faculty}

\begin{document}
\begin{abstract}
We study the VC-dimension of  distance graphs of large subsets of vector spaces over finite fields.  For $3\leq k\leq d$, we prove the distance graph of $E\subset \F_q^d$ has VC-dimension at least $k$ provided that $|E|>Cq^{\min(d-1,\frac{d+k-1}{2})}$.  We also show that this exponent is sharp in the $k=d$ case.  
\end{abstract}
\maketitle

\section{Introduction}

The notion of Vapnik-Chervonenkis dimension (briefly, VC-dimension) is a fundamental concept in statistical learning theory, where it is used to measure the complexity of a learning task.  The precise definition is as follows.

\begin{dfn}
\label{VCdefinition}
Let $X$ be a set, and let $\cH\subset \{0,1\}^X$.  We say $\cH$ \textbf{shatters} a finite set $S\subset X$ if
\[
\{0,1\}^S=\{h|_S:h\in\cH\}.
\]
The \textbf{VC-dimension} of $\cH$ is the largest integer $k$ for which $\cH$ shatters some set of size $k$ (if no such $k$ exists, the VC-dimension is said to be infinite).
\end{dfn}

For a thorough introduction to the theory of VC-dimension and its connection to learning theory, see \cite{ShalevShwartz2014}.  A common variant of Definition \ref{VCdefinition} is the notion of VC-dimension of a graph, which we give here.
\begin{dfn}
\label{def: Graph VC}
Let $G$ be a (not necessarily finite) graph with vertex set $V$ and adjacency relation denoted by $x\sim y$.  The \textbf{neighborhood} of a vertex $x\in V$ is the set
\[
N(x):=\{y\in V:x\sim y\}.
\]
If $G$ is bipartite with parts $X$ and $Y$, we define the VC-dimension of $G$ is defined by taking $X$ as the underlying set and 
\[
\cH_Y:=\{\textbf{1}_{N(y)}:y\in Y\}
\]
as the family of indicator functions.  Finally, we define the VC-dimension of the general graph $G$ as the maximum VC-dimension of an induced bipartite subgraph\footnote{``Induced bipartite subgraph'' means a subgraph with disjoint vertex sets $X$ and $Y$ containing all original edges between $X$ and $Y$, not an induced subgraph which happens to be bipartite.  In other words, we ignore internal edges within $X$ or $Y$.} of $G$. 
\end{dfn}

There are other non-equivalent definitions of VC-dimension of a graph which are commonly used in the literature.  Two common definitions take $X$ to be the entire vertex set, and $\cH$ to be the indicator functions of neighborhoods $N(x)$ or closed neighborhoods $N[x]:=N(x)\cup\{x\}$.  Since each of these notions of VC-dimension is bounded below by the one given in Definition \ref{def: Graph VC}, and since we are interested in lower bounds on the VC-dimension, our main theorem will apply to all of these notions.  

In recent years, there has been considerable interest in obtaining bounds on the VC-dimension of graphs which arise naturally from geometric considerations.  These problems are linked to the study of finite point configurations in sets of sufficient size.  In this paper, we focus on finite field geometries, and we consider distance graphs.  Here and throughout, let $q$ denote a power of an odd prime.

\begin{dfn}
For points $x,y\in\F_q^d$, define the \textbf{distance} by
\[
\|x-y\|=\sum_{i=1}^d (x_i-y_i)^2,
\]
For $t\in\F_q$ and $d\geq 2$, define the \textbf{sphere} of radius $t$ by
\[
S_t^{d-1}=\{x\in\F_q^d:\|x\|=t\}.
\]
Finally, for $E\subset \F_q^d$ and $t\in\F_q$, the \textbf{$t$-distance graph} of $E$ is the graph with vertex set $E$ and adjacency relation defined by
\[
x\sim y \hspace{.25in}\text{if and only if} \hspace{.25in} \|x-y\|=t.
\] 
We denote this graph by $G_t(E)$.
\end{dfn}

The study of the VC-dimension of distance graphs originated with the work of Fitzpatrick, Iosevich, Wyman, and the first listed author \cite{FIMW}.  They prove that if $E\subset \F_q^2$ satisfies $|E|>Cq^{15/8}$, then the VC-dimension of $\hdist(E)$ is 3 for all $t\in\F_q\setminus\{0\}$.  We note that the VC-dimension of $\hdist(\F_q^d)$ is $d+1$, so guaranteeing VC-dimension 3 is best possible (although the exponent $15/8$ is not known to be optimal).  A similar result is obtained for dot products in three dimensions by Iosevich, Sun and the first listed author \cite{VCdot}.  There, they show that if a set $|E|\subset \F_q^3$ satisfies $|E|>Cq^{11/4}$, then the VC-dimension of the analogously defined dot product graph $\gdot(E)$ is 3 for any $t\neq 0$.  Since the VC-dimension of $\hdot(\F_q^d)$ is $d$, this is again the best possible VC-dimension one can guarantee in $d=3$ (but the exponent $11/4$ is not sharp).  The result was generalized to higher dimensions by Ascoli et. al. \cite{Smalldist}, who show that if $|E|>Cq^{s_d}$, where
\[
s_d=
\begin{cases}
7/4, & d=2 \\
7/3, & d=3 \\
d-\frac{1}{d-1}, & d>3,
\end{cases}
\]
then the VC-dimension of $\hdist(E)$ is at least $d$ for any $t\neq 0$.  In the dot product version of the problem, the same authors \cite{Smalldot} show that the VC-dimension of $\hdot(E)$ is $d$ if $|E|>Cq^{d-\frac{1}{d-1}}$.  In the continuous setting, a recent paper of Iosevich, Magyar, and the first and third listed authors \cite{IMMM25} establishes a non-trivial threshold on the Hausdorff dimension of compact $E\subset \R^d$, $d\geq 3$, to obtain VC-dimension at least 3.  See also \cite{DM25} for a generalization which replaces spheres with other Salem sets, and \cite{HMS26, MSW25, RS26} for results on  VC-dimension in other related contexts.

Our main theorem is as follows.

\begin{thm}
\label{thm: final main}
Let $k,d\in\N$ satisfy $3\leq k\leq d$, and define
$$
s(d,k) = \min\left\{
    d-1, \frac{d+k-1}{2}\right\} = \begin{cases}
        \frac{d+k-1}{2}, & k <d \\d-1, & k=d.
    \end{cases}
$$
There exists a constant $C=C_{d,k}$ such that if $E\subset \mathbb{F}_q^d$ satisfies $|E| > Cq^{s(d,k)}$, then for every $t\in \mathbb{F}_q\setminus\{0\}$ on the graph $G_t(E)$ has $VC-$dimension at least k. 
\end{thm}

The extreme cases $k=3$ and $k=d$ deserve special emphasis, since they feature prominently in the literature on this problem.

\begin{cor}
\label{cor: main k=d}
Let $d\geq 3$.  There exists a constant $C=C_d>0$ such that, if $E\subset \F_q^d$ satisfies $|E|>Cq^{d-1}$, then $G_t(E)$ has VC-dimension at least $d$ for every $t\in\F_q\setminus \{0\}$.
\end{cor}
\begin{cor}
\label{cor: main k=3}
Let $d\geq 3$.  There is a constant $C=C_d>0$ such that, if $E\subset \F_q^d$ satisfies $|E|>Cq^{\min(d-1,\frac{d+2}{2})}$, then the VC-dimension of $G_t(E)$ is at least $3$.
\end{cor}
We will also show (Theorem \ref{sharpness of VC threshold}) that the exponent $d-1$ is sharp for the conclusion in Corollary \ref{cor: main k=d}.  For Corollary \ref{cor: main k=3} we do not have a sharpness theorem, but it is interesting to note that the exponent matches the best known exponent for the problem of ensuring the distance graph $G_t(E)$ contains a statistically correct number of $4$-cycles.  

\subsection{Notation and terminology}

\begin{dfn}
For $k\in\N$, let $[k]=\{1,\dots,k\}$.  The \textbf{$k$-shattering graph}, denoted $S_k$, is any of the (isomorphic) bipartite graphs with parts $X_k=\{x_1,\dots,x_k\}$ and $Y_k=\{y_I:I\subset [k]\}$, where $x_i\sim y_I$ if and only if $i\in I$.  
\end{dfn}

The graphs $S_k$ with $k=1,2,3$ are shown in Figure \ref{shattering}.
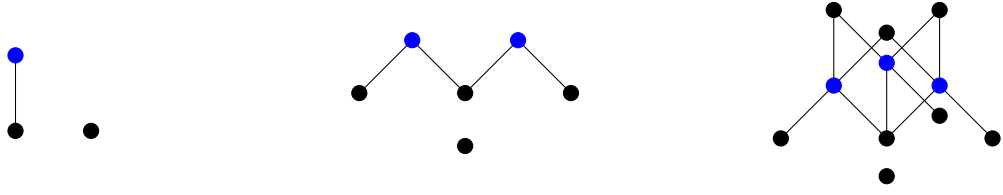
\begin{figure}
\centering
\begin{minipage}{.33\textwidth}
  \centering
\begin{tikzpicture}
\draw(0,0)--(0,1);
\draw[fill](0,0)circle[radius=0.1];
\draw[fill, blue](0,1)circle[radius=0.1];
\draw[fill](1,0)circle[radius=0.1];
\end{tikzpicture} 
\end{minipage}%
\begin{minipage}{.33\textwidth}
  \centering
\begin{tikzpicture}
\draw(-1.4,0)--(-.7,.7)--(0,0)--(.7,.7)--(1.4,0);
\draw[fill](0,0)circle[radius=0.1];
\draw[fill](-1.4,0)circle[radius=0.1];
\draw[fill](1.4,0)circle[radius=0.1];
\draw[fill, blue](.7,.7)circle[radius=0.1];
\draw[fill, blue](-.7,.7)circle[radius=0.1];
\draw[fill](0,-.7)circle[radius=0.1];
\end{tikzpicture}
\end{minipage}
\begin{minipage}{.33\textwidth}
  \centering
\begin{tikzpicture}
\draw(0,0)--(.7,.7);
\draw(0,0)--(0,1);
\draw(0,0)--(-.7,.7);
\draw(.7,.7)--(.7,1.7)--(0,1);
\draw(.7,.7)--(0,1.4)--(-.7,.7);
\draw(-.7,.7)--(-.7,1.7)--(0,1);
\draw(-.7,.7)--(-1.4,0);
\draw(.7,.7)--(1.4,0);
\draw(0,1)--(.7,.3);

\draw[fill](0,0)circle[radius=0.1];

\draw[fill](.7,1.7)circle[radius=0.1];
\draw[fill](-.7,1.7)circle[radius=0.1];
\draw[fill](0,1.4)circle[radius=0.1];

\draw[fill](-1.4,0)circle[radius=0.1];
\draw[fill](1.4,0)circle[radius=0.1];
\draw[fill](.7,.3)circle[radius=0.1];

\draw[fill, blue](.7,.7)circle[radius=0.1];
\draw[fill, blue](-.7,.7)circle[radius=0.1];
\draw[fill, blue](0,1)circle[radius=0.1];

\draw[fill](0,-.5)circle[radius=0.1];
\end{tikzpicture}
\end{minipage}
\caption{The graphs $S_1,S_2,S_3$ ($X_k$ blue, $Y_k$ black).}
  \label{shattering}
\end{figure}
It is clear that the VC-dimension of a graph $G$ is at least $k$ if and only if $G$ contains $S_k$ as a bipartite induced subgraph; that is, if and only if there exist sets of vertices $\{x_i\}_i,\{y_I\}_I$ such that $x_i\sim y_I$ if and only if $i\in I$.  Our goal will be to construct such a configuration consisting of points in a set $E\subset \F_q^d$, where $\sim$ denotes adjacency in the graph $G_t(E)$ (more generally, throughout this paper, graph theoretic notation will always be defined with respect to the distance graph unless indicated otherwise).  

\section{Intersections of affine subspaces and spheres in $\F_q^d$}
Throughout this section, it will be convenient to adopt a notational convention in order to distinguish between tuples of vectors and the coordinates of individual vectors.  We will use superscripts to denote vectors in a sequence and subscipts to denote coordinates.  Thus, we will write things like $v^1,\dots,v^n$ for a sequence of $n$ vectors $v^i$, and $v_j$ for the $j$th coordinate of a given vector $v$.
\begin{dfn}
Let $V$ be a finite dimensional vector space over a field $\F$.  A \textbf{bilinear form} on $V$ is a map $B:V\times V\to \F$ which is linear in each variable separately.  The \textbf{kernel} of the form is the set
\[
\ker B:=\{v\in V:(\forall w\in V)\ B(v,w)=0\}.
\]
A bilinear form is \textbf{degenerate} if $\ker B\neq \{0\}$.  The form is \textbf{symmetric} if $B(v,w)=B(w,v)$ for all $v,w\in V$.  For a symmetric bilinear form, we say $v,w\in V$ are \textbf{orthogonal} if $B(v,w)=0$, and write
\[
W^\perp=\{v\in V:(\forall w\in W)\ B(v,w)=0\}
\]
for the orthogonal subspace with respect to $B$.
\end{dfn}
We use the following standard facts about the structure of symmetric bilinear forms (see, for example, \cite[Chapter XV]{Lang}).
\begin{thm}[Basic properties of symmetric bilinear forms]
\label{thm: bilinear properties}
Let $\F$ be a field of characteristic not equal to $2$, let $V$ be a vector space over $\F$ of finite positive dimension, and let $B$ be a symmetric bilinear form on $V$.  
\begin{enumerate}[(a)]
    \item There exists a subspace $W\subset V$ such that $B$ restricts to a non-degenerate form on $W$, and $V=W\oplus \ker B$.
    \item $V$ has an orthogonal basis with respect to $B$.
    \item If $B$ is non-degenerate, then for any subspace $W\subset V$ we have $\dim W+\dim W^\perp=\dim V$.
    \item If $B$ is non-degenerate, for any subspace $W\subset V$, we have $V=W+W^\perp$ if and only if $B$ restricts to a non-degenerate form on $W$.
\end{enumerate}
\end{thm}

\begin{dfn}
Let $\F$ be a field, $d\geq 1$.  A \textbf{quadratic form} over $\F^d$ is a function $Q:\F^d\to \F$ of the form
\[
Q(x)=\sum_{1\leq i,j\leq d}a_{i,j}x_ix_j
\]
for some coefficients $a_{i,j}\in \F$.  A subspace $V\subset \F^d$ is called \textbf{isotropic} (with respect to $Q$) if there exists $x\in V\setminus\{0\}$ such that $Q(x)=0$, and is called \textbf{totally isotropic} if $Q(x)=0$ for all $x\in V$.
\end{dfn}

\begin{thm}[Orthogonality and totally isotropic subspaces]
\label{thm: totally isotropic properties}
Let $\F$ be a field of characteristic not equal to 2, let $B$ be a non-degenerate symmetric bilinear form on $\F^d$, and let $Q$ be the associated quadratic form (that is, let $Q(x)=B(x,x)$).
\begin{enumerate}[(a)]
\item If $V\subset \F^d$ is a totally isotropic subspace, then $V\subset V^\perp$.
\item If $V\subset \F^d$ is totally isotropic, then $\dim V\leq \frac{d}{2}$.
\item If $V\subset \F^d$ is a totally isotropic subspace and $\dim V=\frac{d}{2}$, then $V=V^\perp$.
\end{enumerate}
\end{thm}
\begin{proof}
Let $\{v^1,\dots,v^k\}$ be an orthogonal basis for $V$.  For any $v\in V$, write $v=\sum_i a_iv^i$.  For any other sequence of coordinates $\{b_i\}$, we have
\[
B\left(v, \sum_{j=1}^k b_iv^i\right)=\sum_{i,j}a_ib_jB(v^i, v^j).
\]
Since the basis is orthogonal, $B(v^i,v^j)=0$ whenever $i\neq j$.  Since $V$ is totally isotropic, we have $B(v^i,v^i)=0$ also.  It follows that $B(v,w)=0$ for all $w\in V$, hence $v\in V^\perp$.  This proves (a).  Parts (b) and (c) follows from (a) and the relation  $\dim V+\dim V^\perp=d$ given in Theorem \ref{thm: bilinear properties}.
\end{proof}

We will be interested in the case $\F=\F_q$ ($q$ odd).  Here, the length $\|\cdot\|$ is a quadratic form, and the corresponding bilinear form is the dot product.  If $V\subset \F_q^d$ is a subspace we will refer to $V$ as ``isotropic'' (respectively, ``totally isotropic'') if the restriction of $\|\cdot \|$ to $V$ is isotropic (respectively, totally isotropic) and ``degenerate'' if the restriction of the dot product to $V\times V$ is degenerate.  We have the following bound.
\begin{lem}
\label{lem: non-degen sphere}
Let $d\in\N$, and let $Q$ be a non-degenerate quadratic form on $\F_q^d$.  For any $t\in\F_q$, we have
\[
|\{x\in \mathbb{F}_q^d:Q(x)=t\}|\leq 2q^{d-1}.
\]
\end{lem}
This is a standard result; see \cite[Theorem 7.1]{BHIPR} for an exact formula for the left hand side.

\begin{thm}
\label{thm: sphere intersect V}
Let $V\subset \F_q^d$ be a subspace of dimension $k\geq 1$ which is not totally isotropic.  Then, for any $t\in\F_q$, we have
\[
|V\cap S_t^{d-1}|\leq 2q^{k-1}.
\]
\end{thm}
\begin{proof}
By Theorem \ref{thm: bilinear properties}, we may fix a basis $\{v^1,\dots,v^k\}$ for $V$ which is orthogonal with respect to the dot product.  Without loss of generality, suppose we have ordered our basis so that $\|v^i\|=0$ if and only if $i>l$.  Because $V$ is not totally isotropic, we have $1\leq l\leq k$.  Define a quadratic form $Q$ on $\F_q^{l}$ by
\[
Q(x)=\|x_1v^1+\cdots+x_{l}v^{l}\|=\sum_{i=1}^{l} \|v^i\|x_i^2.
\]
The associated bilinear form is 
\[
B(x,y)=\sum_{i=1}^{l}\|v^i\|x_iy_i.
\]
If $e^i$ denotes the $i$-th standard basis vector, then $B(x,e^i)=\|v^i\|x_i$.  This implies that $B$ is non-degenerate, since $\|v^j\|\neq 0$ for $1\leq j\leq l$.  Finally, we note that for any $v=x_1v^1+\cdots+x_kv^k\in V$, we have $\|v\|=Q(x_1,\dots,x_{l})$.  Thus, coordinates $(x_1,\dots,x_k)\in \F_q^k$ represent a point on $S_t^{d-1}$ if and only if the first $l$ coordinates $x:=(x_1,\dots,x_{l})$ satisfy $Q(x)=t$; coordinates $(x_{l+1},\dots,x_k)$ can be arbitrary.  By Lemma \ref{lem: non-degen sphere}, we have
\[
|V\cap S_t^{d-1}|=q^{k-l}|\{x\in \F_q^{l}:Q(x)=t\}| \leq 2q^{k-1 }.
\]
\end{proof}
\begin{thm}
\label{thm: gen sphere intersect V}
Let $V\subset \F_q^d$ be a subspace of dimension $k\geq 1$.  Let $S$ be any sphere (that is, let $S=z+S_t^{d-1}$ for some $z\in\F_q^d,t\in\F_q$).  
\begin{enumerate}[(a)]
    \item If $V$ is is not totally isotropic with respect to the length $\|\cdot\|$, then
    \[
    |S\cap V|\leq 2q^{k-1}.
    \]
    \item If $V$ is totally isotropic and $2k=d$, and the radius of $S$ is non-zero, then
    \[
    |S\cap V|\leq q^{k-1}.
    \]
\end{enumerate}
\end{thm}
\begin{proof}
\textbf{(a):}  First, consider the case where $V$ is non-degenerate.  By Theorem \ref{thm: bilinear properties}, we have $\F_q^d=V+V^\perp$.  We may therefore write $z=v+v^\perp$ with $v\in V$ and $v^\perp\in V^\perp$.  We observe that $\|z\|=\|v\|+\|v^\perp\|$, and that for $x\in V$, we have $x\cdot z=x\cdot v$, hence
\begin{align*}
\|x-z\|&=\|x\|+\|z\|-2x\cdot z \\
&=\|x\|+\|v\|+\|v^\perp\|-2x\cdot v \\
&=\|x-v\|+\|v^\perp\|.
\end{align*}
It follows that
\[
V\cap (z+S_t^{d-1})=V\cap (v+S_{t-\|v^\perp\|}^{d-1}).
\]
Since $x\mapsto x-v$ is a bijection on $V$, we therefore have
\begin{align*}
|V\cap S|&=\left|V\cap (v+S_{t-\|v^\perp\|}^{d-1})\right| \\
&=\left|(V-v)\cap S_{t-\|v^\perp\|}^{d-1}\right| \\
&=\left|V\cap S_{t-\|v^\perp\|}^{d-1}\right| \\
&\leq2q^{k-1}
\end{align*}
by Theorem \ref{thm: sphere intersect V}.

Next, we consider the case where $V$ is degenerate.  By Theorem \ref{thm: bilinear properties}, we may write $V=W\oplus V_0$, where $V_0$ is the kernel of the restriction of the dot product to $V$, and $W$ is non-degenerate.  Since $V$ is not totally isotropic, $W$ is non-trival.  If $v=w+u$ for $w\in W,u\in V_0$, then
\begin{align*}
\|v-z\|&=\|v\|-2v\cdot z+\|z\| \\
&=\|w\|-2w\cdot z+\|z\|-2u\cdot z \\
&=\|w-z\|-2u\cdot z.
\end{align*}
Therefore, $\|v-z\|=t$ if and only if $\|w-z\|=t+2u\cdot z$.  Let $m=\dim W$.  By the non-degenerate case, we have
\begin{align*}
|\{v\in V:\|v-z\|=t\|\}|&=\sum_{u\in V_0}|\{w\in W:\|w-z\|=t+2u\cdot z\}| \\
&\leq 2q^{m-1}|V_0| \\
&=2q^{k-1}.
\end{align*}

\textbf{(b):}  Since $V$ is totally isotropic, by Theorem \ref{thm: totally isotropic properties} we have $V\subset V^\perp$.  By Theorem \ref{thm: bilinear properties} and the assumption $\dim V=\frac{d}{2}$ we also have $\dim V=\dim V^\perp$, hence $V=V^\perp$ holds.  We are counting solutions to the equation
\begin{equation}
\label{eq: tot iso sphere I}
t=\|x\|+\|z\|-2x\cdot z,\hspace{.25in} x\in V.
\end{equation}
Since $V$ is totally isotropic, $x\in V$ implies $\|x\|=0$, hence the number of solutions to (\ref{eq: tot iso sphere I}) is bounded above by the number of solutions to the equation
\begin{equation}
\label{eq: tot iso sphere II}
x\cdot z=\frac{\|z\|-t}{2},\hspace{.25in} x\in V.
\end{equation}
If $z\in V$ then (\ref{eq: tot iso sphere II}) has no solutions, as the left hand side is 0 and the right hand side is $-t/2$ (which is assumed non-zero).  If $z\notin V$, then the number of solutions to (\ref{eq: tot iso sphere II}) is the same as the number of solutions to the homogenized equation $x\cdot z=0$ with $x\in V$.  If all $x\in V$ are solutions then $z\in V^\perp$, hence $z\in V$.  Therefore, the solution set is a proper subspace, and hence contains at most $q^{k-1}$ points.
\end{proof}

If $A\subset \F_q^d$ is an affine subspace, there is a unique vector space which is a translate of $A$; denote this space by $V(A)$.  The following is immediate.
\begin{cor}
\label{cor: Sphere intersect A}
Let $A\subset \F_q^d$ be an affine subspace, let $k=\dim V(A)$, and let $S$ be a sphere.
\begin{enumerate}[(a)]
    \item If $V(A)$ is not totally isotropic, then
    \[
    |S\cap A|\leq 2q^{k-1}.
    \]
    \item If $V(A)$ is totally isotropic and $2k=d$, and the radius of $S$ is non-zero, then
    \[
    |S\cap A|\leq q^{k-1}.
    \]
\end{enumerate}
\end{cor}

Combining Corollary \ref{cor: Sphere intersect A} with the maximum size of a totally isotropic subspace (Theorem \ref{thm: totally isotropic properties}), we get the following.
\begin{cor}
\label{cor: Sphere intersect high dim affine}
Let $r\geq \frac{d+2}{2}$, let $x^1,\dots,x^r$ be affinely independent, and let $A$ be the affine span of these points.  For any sphere $S$ of non-zero radius, we have
\[
|S\cap A|\leq 2q^{r-2}.
\]
Moreover, if $r>\frac{d+2}{2}$, the assumption that the radius is non-zero may be removed.
\end{cor}

Finally, we observe that counting intersections of spheres and affine subspaces allows us to count intersections of spheres.

\begin{thm}
\label{thm: few spheres intersection}
Let $1\leq r\leq\frac{d+2}{2}$.  If $z^1,\dots,z^r\in\F_q^d$ are affinely independent vectors, then for any $t_1,\dots,t_r\in\F_q$ not all zero, we have
\[
\left|\bigcap_{i=1}^{r} (z^i+S_{t_i}^{d-1})\right|\leq 2q^{d-r}.
\]
Moreover, if $r<\frac{d+2}{2}$, the assumption that the radii are not all zero may be removed.
\end{thm}
\begin{proof}
Assume without loss of generality that $t_r\neq 0$.  We are counting solutions $x$ to the system of equations
\begin{equation}
\label{eq: sphere system 1}
\begin{split}
\|x-z^1\|&=t_1, \\
&\vdots \\
\|x-z^r\|&=t_r.
\end{split}
\end{equation}
Making the change of variables $y=x-z^r,w^i=z^i-z^r$, (\ref{eq: sphere system 1}) is equivalent to the system
\begin{equation}
\label{eq: sphere system 2}
\begin{split}
\|y-w^1\|&=t_1, \\
&\vdots \\
\|y-w^{r-1}\|&=t_{r-1}, \\
\|y\|&=t_r.
\end{split}
\end{equation}
Let $A$ be the $(r-1)\times d$ matrix with rows $w^1,\dots,w^{r-1}$, and let
\[
b=\frac{1}{2}
\begin{pmatrix}
\|w^1\|+t_r-t_1 \\
\vdots \\
\|w^{r-1}\|+t_r-t_{r-1}
\end{pmatrix}.
\]
In this notation, (\ref{eq: sphere system 2}) is equivalent to the pair of equations
\begin{align}
\label{eq: matrix}
Ay=b, \\
\label{eq: last sphere}
\|y\|=t_r.
\end{align}
The solution set of (\ref{eq: matrix}) is either empty, or else has the same number of solutions as the homogenized equation $Ax=0$.  Since the rows of $A$ are linearly independent, we have $\text{rank}\: A=r-1$.  Together with the assumption $r\leq \frac{d+2}{2}$, this implies
\[
\dim\ker A=d-(r-1)\geq \frac{d}{2}.
\]
If the inequality is strict then $\ker A$ is not totally isotropic.  If equality holds, then $\ker A$ has dimension $\frac{d}{2}$.  In either case, we may apply Corollary \ref{cor: Sphere intersect A} to the solution space of (\ref{eq: matrix}).  Since $t_r\neq 0$, we may apply either part (a) or (b) of that corollary to conclude that the number of solutions to the system (\ref{eq: matrix})-(\ref{eq: last sphere}) is at most $2q^{d-r}$.  If $r<\frac{d+2}{2}$ then $\dim\ker A>\frac{d}{2}$, hence we may apply part (a) without assuming $t_r\neq 0$, giving the ``moreover'' part of the statement.
\end{proof}

\section{Proof of Theorem \ref{thm: final main}}

\subsection{Lemmas}
We start with a standard Fourier-analytic estimate which is frequently useful in the study of distance graphs.
\begin{lem}[\protect{\cite[Theorem 2.1]{BCCHIP16}}]
\label{functional IRlemma}
Let $f,g:\F_q^d\to [0,\infty)$.  For any $t\in\F_q\setminus\{0\}$, we have
\[
\sum_{x,y\in \F_q^d}f(x)g(y)S_t^{d-1}(x-y)=\frac{|S_t^{d-1}|}{q^d}\|f\|_{L^1}\|g\|_{L^1}+R_t(f,g),
\]
where the remainder satisfies the bound
\[
|R_t(f,g)|\leq 2q^{\frac{d-1}{2}}\|f\|_{L^2}\|g\|_{L^2}.
\]
\end{lem}

As a consequence of Lemma \ref{functional IRlemma}, we see that we may always assume without loss of generality that our distance graphs are almost regular.  More precisely, we have the following.

\begin{lem}
\label{distance graph almost regular}
Let $E\subset \F_q^d$.  There exists $C=C_d>0$ such that if $E>Cq^\frac{d+1}{2}$, then for any $t\in\F_q\setminus \{0\}$ there exists a subset $E'\subset E$ such that $|E'|\approx |E|$, and for every $x\in E'$ we have 
    \[
    \deg(x)\approx \frac{|E|}{q}.
    \]
\end{lem}
\begin{proof}
Let
\[
E_{\text{low}}=\left\{x\in E:\deg(x)<\frac{|E|}{2q}\right\}
\sep
E_{\text{high}}=\left\{x\in E:\deg(x)>\frac{2|E|}{q}\right\}.
\]
We have
\[
\sum_{x,y\in\F_q^d}E_{\text{low}}(x)E(y)S_t^{d-1}(x-y)
=\sum_{x\in E_{\text{low}}}\deg(x) \leq \frac{|E_{\text{low}}||E|}{2q}.
\]
On the other hand, if $t\neq 0$, Lemma \ref{functional IRlemma} implies
\begin{align*}
\sum_{x,y\in\F_q^d}E_{\text{low}}(x)E(y)S_t^{d-1}(x-y)&=\frac{|S_t^{d-1}|}{q^d}\|E_{\text{low}}\|_{L^1}\|E\|_{L^1}+R_t(E_{\text{low}},R) \\
&=\frac{|S_t^{d-1}|}{q^d}|E_{\text{low}}||E|+R_t(E_{\text{low}},E),
\end{align*}
where
\[
|R_t(E_{\text{low}},E)|\leq 2q^{\frac{d-1}{2}}\|E_{\text{low}}\|_{L^2}\|E\|_{L^2}=2q^\frac{d-1}{2}|E_{\text{low}}|^{1/2}|E|^{1/2}.
\]
Putting this together, we have
\[
\left(\frac{|S_t^{d-1}|}{q^{d-1}}-\frac{1}{2}\right)\frac{|E_{\text{low}}||E|}{q}<2q^{\frac{d-1}{2}}|E_{\text{low}}|^{1/2}|E|^{1/2}.
\]
Using the standard estimate $|S_t^{d-1}|=q^{d-1}(1+o(1))$, 
this implies
\[
|E_{\text{low}}||E|\lesssim q^{d+1}.
\]
The assumption $|E|>q^\frac{d+1}{2}$ then implies
\[
|E_{\text{low}}|\lesssim q^{\frac{d+1}{2}}.
\]
A similar argument shows $|E_{\text{high}}|\lesssim q^{\frac{d+1}{2}}$ also.  Therefore, if $|E|>Cq^{\frac{d+1}{2}}$ with $C$ sufficiently large, we may take
\[
E'=E\setminus(E_{\text{low}}\cup E_{\text{high}})
\]
to complete the proof.
\end{proof}

\begin{dfn}
Fix $E\subset \F_q^d$ and $t\in\F_q\setminus \{0\}$.  A \textbf{star} is a copy of $K_{k,1}$ in $G_t(E)$, which we will denote by $(X,y)$, where $X\in \binom{E}{k}$ is the $k$-part and $y\in E$ is the $1$-part.
\end{dfn}

To prove our theorem, we start by estimating the number of $k$-stars in $E$.  We will think of such a star as a building block for a copy of the shattering graph $S_k$, where we take $X$ as the $k$ points being shattered and $y$ as our point $y_{[k]}$.  For technical reasons, we will also need to keep track of the affine rank of the leaves.

\begin{lem}
\label{lem: affine stars}
Fix $t\in \mathbb{F}_q\setminus \{0\}$.  Let $d\geq 3$ let $E,E'\subset \F_q^d$, and for $1\leq k\leq d$ and $t\in\F_q\setminus \{0\}$ let $\cA_{t,k}(E,E')$ denote the set of $k$-stars $(X,y)$ for which $X\subset E$, $y\in E'$, and $X$ is affinely independent.  There exists $C>0$ such that if $|E|>Cq^{\max(\frac{d+1}{2},k-2)}$, then there exists $E'\subset E$ such that $|E'|\approx |E|$, and
\[
|\cA_{t,k}(E,E')|\approx \frac{|E|^{k+1}}{q^k}.
\]
\end{lem}
\begin{proof}
Assume $C$ is at least as large as the constant in Lemma \ref{distance graph almost regular}, and $E'$ be as in that Lemma.  This implies that any $y\in E'$ is the center of $\approx \frac{|E|^{k}}{q^k}$ stars $(X,y)$ with $X\subset E$.  This establishes the theorem for $k=1,2$, since any two points are affinely independent.  To prove the theorem for $k\geq 3$, we must show that the number of affinely dependent stars is a small proportion of the total.  We start by establishing some notation.  Given affinely independent points $x_1,\dots,x_r$, let $A_{x_1,\dots,x_r}$ denote the affine span.  For $2\leq r<k$, let $\cA_{t,k}^r(E,E')$ be the set of $k$ stars $(X,y)$ such that $r$ is the maximum number of affinely independent points in $X$.  We have
\begin{equation}
\label{eq: count dep stars}
\begin{split}
|\cA_{t,k}^r(E,E')|&\lesssim \sum_{y\in E'}\sum_{\substack{x_1,\dots,x_r\in N(y) \\ \text{affinely independent}}}\sum_{x_{r+1},\dots,x_k\in N(y)\cap A_{x_1,\dots,x_r}}1 \\
&\leq \sum_{y\in E'}\sum_{\substack{x_1,\dots,x_r\in N(y) \\ \text{affinely independent}}} |(S_t^{d-1}+y)\cap A_{x_1,\dots,x_r}|^{k-r}.
\end{split}
\end{equation}
If $r\geq \frac{d+2}{2}$, then for any sphere $y$ we have by Corollary \ref{cor: Sphere intersect high dim affine}
\[
|(S_t^{d-1}+y)\cap A_{x_1,\dots,x_r}|\leq 2q^{r-2}.
\]
Plugging this into (\ref{eq: count dep stars}) yields
\begin{equation}
\label{eq: count dep stars high}
\begin{split}
|\cA_{t,k}^r(E,E')|&\lesssim\sum_{y\in E'}\sum_{\substack{x_1,\dots,x_r\in N(y) \\ \text{affinely independent}}} q^{(r-2)(k-r)} \\[.1in]
&\lesssim \frac{|E|^{r+1}}{q^r}q^{(r-2)(k-r)} \\
&=|E|^{r+1}q^{r(k-r)-(2k-r)},
\end{split}
\end{equation}
provided $r\geq \frac{d+2}{2}$.  The bound in (\ref{eq: count dep stars high}) is dominated by $\frac{|E|^{k+1}}{q^k}$ if $|E|>Cq^{r-1}$ for a sufficiently large constant $C$.  Since $r<k$, this follows if $|E|>Cq^{k-2}$.  On the other hand, if $r< \frac{d+2}{2}$, then for any sphere $S$ we plug the trivial bound $|A_{x_1,\dots,x_r}|\leq q^{r-1}$ into (\ref{eq: count dep stars}) to get
\begin{equation}
\label{eq: count dep stars low}
\begin{split}
|\cA_{t,k}^r(E,E')|&\lesssim\sum_{y\in E'}\sum_{\substack{x_1,\dots,x_r\in N(y) \\ \text{affinely independent}}} q^{(r-1)(k-r)} \\[.1in]
&\lesssim \frac{|E|^{r+1}}{q^r}q^{(r-1)(k-r)} \\
&=|E|^{r+1}q^{(r-1)(k-r)-r}.
\end{split}
\end{equation}
The bound in (\ref{eq: count dep stars low}) is dominated by $\frac{|E|^{k+1}}{q^k}$ whenever $|E|>q^r$.  Since $r<\frac{d+2}{2}$, this follows from the assumption $|E|>Cq^{\frac{d+1}{2}}$.

\end{proof}

The next estimate bounds the number of vertices that are adjacent to a fixed positive fraction of a given set. 
\begin{lem}\label{degree-bounding}
    Let $S \subset E$ be nonempty, $p = \frac{|S_t^{d-1}|}{q^d}$. Then for $\alpha > p$,

    $$
    T_\alpha(S) = \{x \in E: |N(x) \cap S| \ge \alpha|S|\} \text{ satisfies} \quad |T_\alpha(S)| \le\frac{4q^{d-1}}{(\alpha-p)^2|S|}.
    $$ 
\end{lem}

\begin{proof}
Fix $S$ and $\al$, and let $T=T_\al(S)$.  By Lemma \ref{functional IRlemma}, we have
    $$
    \alpha|T||S| \le e(T,S) = \sum_{x,y\in \F_q^d}T(x)S(y)S_t^{d-1}(x-y) \le p|T||S| + 2q^{(d-1)/2}\sqrt{|T||S|}.
    $$
After rearranging, 
$$
|T|\le \frac{4q^{d-1}}{\left(\alpha - p \right)^2|S|}
$$
\end{proof}

\subsection{Bad sets.}
We now turn to the proof of our main theorem.  The idea is to combine our estimate for affinely independent stars (Lemma \ref{lem: affine stars}) with an upper bound on the number of stars which cannot be used to build a copy of $S_k$.  In order to do this, we first examine what such an obstacle would look like.  This leads us naturally to consider intersections of neighborhoods of several points.  We will use the following notation.

\begin{dfn}
If $G$ is a graph and $x$ is a vertex, recall that $N(x)$ denotes the neighborhood of $x$.  If $S$ is a set of vertices, we define
\[
N(S)=\bigcup_{x\in S}N(x)
\]
and
\[
\neigh(S)=\bigcap_{x\in S}N(x).
\]
\end{dfn}

\begin{dfn}
    Let $E\subset \F_q^d,t\in\F_q\setminus \{0\}$.  If $(X,y)$ is a $k$-star, we say a set $B\subset X$ is \textbf{$(X,y)$-bad} if $0\leq |B|<k$ and
\[
\neigh(B)\subset X \cup  N(X\setminus B).
\]
Finally, we say the star $(X,y)$ \textbf{admits a bad set} if there exists $B\subset X$ which is $(X,y)$-bad.
\end{dfn}

We are now ready to prove our main theorem.  The idea of the proof is to use a counting argument to exhibit a star which does not admit a bad set.  The proof is slightly different in the cases $k<d$ and $k=d$, so we handle each separately.

\begin{proof}[Proof of Theorem \ref{thm: final main}, $k<d$ case]
When $k<d$, we have $\min(d-1,\frac{d+k-1}{2})=\frac{d+k-1}{2}$.  It is enough to show that when $|E|>Cq^{\frac{d+k-1}{2}}$, there exists a $k$-star which does not admit a bad set.  To see this, suppose $(X,y)$ is such a star.  We start by defining $y_{[k]}=y$.  For any proper subset $I\subset [k]$, since the set $\{x_i:i\in I\}$ is not bad, there exists a point 
\[
z\in \neigh(\{x_i:i\in I\})\setminus \left(X \cup N(\{x_i:i\notin I\})\right).
\]
In other words, we have $z\sim x_i$ if and only if $i\in I$, hence we may take $y_I=z$.

To show that such a $k$-star exists, let $E'$ be as in Lemma \ref{distance graph almost regular}, and let $E''$ be as in Lemma \ref{lem: affine stars} with respect to this set $E'$.  Let $M$ be the number of stars in $\cA_{t,k}(E',E'')$ which admit a bad set.  We shall prove $M<|\cA_{t,k}(E',E'')|$.  For $B\subset E'$ with $0\leq |B|<k$, let $M(B)$ denote the number of $k$-stars $(X,y)\in\cA_{t,k}(E',E'')$ so that $B$ is $(X,y)$-bad (with respect to the ambient set $E'$).  We first note that $\emp$ is not bad for any $(X,y)\in \cA_{t,k}(E',E'')$.  Indeed, if $\emp$ is bad for $(X,y)$ then $E'\subset X \cup N(X)$, but by choice of $E'$ we have
\[
|N(X)\cup X|\leq |X| + |N(X)| \lesssim k +\frac{2k|E|}{q}< |E|
\]
for $q$ and $|E|$ sufficiently large.

We therefore have
\begin{equation}
\label{eq: decompose M(B)}
M\le\sum_{j=1}^{k-1}\sum_{B\in \binom{E}{j}}M(B)\approx \max_{1\leq j<k}\sum_{B\in \binom{E}{j}}M(B).
\end{equation}

Fix $B\subset E$ with $|B|=j$, let $h= k-j$, and suppose $M(B)\neq 0$.  We consider two cases.

\textbf{Case I:} $|\neigh(B)| > 2k$.  Let $(X,y)$ be a star for which $B$ is bad. Every point of $\neigh(B)\setminus X$ is adjacent to some point in $X\setminus B$, so
\[
\sum_{x\in X\setminus B}|N(x)\cap \neigh(B)|\geq |\neigh(B)|-h>\frac{1}{2}|\neigh(B)|.
\]
By the pigeonhole principle, there exists $x\in X\setminus B$ such that $|N(x) \cap \neigh(B)| \ge|\neigh(B)|/2h$.  By Lemma \ref{degree-bounding}, there are $O(q^{d-1}/|\neigh(B)|)$ choices for this vertex.  Combining this with $|\neigh(B)|$ choices for $y$ and $\frac{|E|^{h-1}}{q^{h-1}}$ choices for the remaining $h-1$ vertices of $X\setminus B$ (after $y$ has been fixed), we have

$$
    M(B) \lesssim |\neigh(B)| \frac{q^{d-1}}{|\neigh(B)|} \frac{|E|^{h-1}}{q^{h-1}} \le q^{d-1} \frac{|E|^{h-1}}{q^{h-1}} .
$$
Let $M_{I,j}$ denote the sum on the right hand side of (\ref{eq: decompose M(B)}), summed over $B\in\binom{E}{j}$ in Case I only.  We have
$$
    M_I \lesssim {|E|^j} q^{d-1} \frac{|E|^{k-j-1}}{q^{k-j-1}} = |E|^{k-1} q^{d+j-k} .
$$

\textbf{Case II:} $|\neigh(B)| \le 2k$.  To build a star $(X,y)$ for which $B$ is bad, we have $|\neigh(B)|=O(1)$ choices for $y$ followed by $\frac{|E|^h}{q^h}$ choices for each point of $X\setminus B$, giving
$$
M(B) \lesssim \frac{|E|^h}{q^h}.
$$
Let $M_{II,j}$ denote the sum on the right hand side of (\ref{eq: decompose M(B)}), summed over $B\in\binom{E}{j}$ in Case II only.  We have

\begin{equation}\label{small nbhd}
    M_{II,j}\lesssim |E|^j |E|^{k-j}q^{j-k} = |E|^kq^{j-k}.
\end{equation}

Combining Cases I and II, we have
\[
M\lesssim \max_{1\leq j<k}M_{I,j}+M_{II,j}\lesssim \max_{1\leq j<k}\left(|E|^{k-1} q^{d+j-k}+|E|^kq^{j-k}\right)\approx |E|^{k-1} q^{d-1}.
\]
Since $\cA_{t,k}(E',E'')\approx \frac{|E|^{k+1}}{q^k}$, we have $M<\cA_{t,k}(E',E'')$ provided $|E|>Cq^{\frac{d+k-1}{2}}$ for a sufficiently large constant $C$.
\end{proof}
\begin{proof}[Proof of Theorem \ref{thm: final main}, $k=d$ case]
In this case, $\min(d-1,\frac{d+k-1}{2})=d-1$.  The proof follows the same strategy as in the $k<d$ case, but we use a different estimate to handle Case I.  We define $M,M_{I,j},M_{II,j},M(B)$ as before.  Assume $|\neigh(B)|>2k$.  At most $h\le k-1$ points of $\neigh(B)$ lie in $X$, so if $B$ is bad  at least $k+2>h$ points of $\neigh(B)\setminus X$ are adjacent to some point in $X\setminus B$. By the pigeonhole principle, some $x\in X\setminus B$ is adjacent to two distinct points $z,z' \in \neigh(B)$.  By Theorem \ref{thm: few spheres intersection} with $r=2$, for each pair of points $\{z,z'\} \subset \tilde{N}(B)$ there are at most $2q^{d-2}$ such points. Choosing the center $y\in \neigh(B)$, then the pair $z,z'$, then $x$, then the remaining $h-1$ points in $N(y)$ gives

$$
M(B) \lesssim |\neigh(B)|q^{d-2}\frac{|E|^{h-1}}{q^{h-1}}=|\neigh(B)|q^{d-2}\frac{|E|^{k-j-1}}{q^{k-j-1}}.
$$
We also have
$$
\sum_{B\in\binom Ej}|\neigh(B)|=\sum_{y\in E'}\binom{\deg_E(y)}{j}\leq |E|'(|E|/q)^j\lesssim \frac{|E|^{j+1}}{q^j},
$$
hence
$$
\frac{1}{|\mathcal{A}|}\sum_{\substack{B\in {E'\choose j}\\|\neigh(B)|>2k}} M(B) \lesssim \frac{q^{d-2}(|E|/q)^{k-j-1}|E|^{j+1}q^{-j}}{|E|^{k+1}q^{-k}} = \frac{q^{d-1}}{|E|}.
$$

Combined with (\ref{small nbhd}) and $j\le k-1\le d-1,$ this gives $M/|\mathcal{A|} \lesssim q^{d-1}/|E|,$ which is less than 1 if $|E| > Cq^{d-1}$.
\end{proof}


\section{Sharpness of Corollary \ref{cor: main k=d}}
The purpose of this section is to show that the exponent $d-1$ in Corollary \ref{cor: main k=d} is sharp.  We note that the VC-dimension of $G_t(\F_q^d)$ is $d+1$, and we do not address the issue of what exponent is needed to ensure $G_t(E)$ has VC-dimension $d+1$.  We also note that we only show sharpness with respect to the strong conclusion required in this paper, where a certain exponent guarantees $G_t(E)$ has the claimed VC-dimension for all nonzero $t$.  If we instead ask for an exponent $s$ such that $|E|>Cq^s$ implies $G_t(E)$ has VC-dimension at least $d$ for some $t$, or a positive proportion of $t$, the question remains open.  We also note that we use a definition of VC-dimension (Definition \ref{def: Graph VC}) which requires the witnessing vertices $\{x_i\}_i$ and $\{y_I\}_I$ to be disjoint.  We noted that other common definitions of the VC-dimension of a graph which use open or closed neighborhoods are bounded below by this version of VC-dimension, so Theorem \ref{thm: final main} applies to all three VC-dimensions.  We note here that our sharpness theorem applies only to the notion of VC-dimension given in Definition \ref{def: Graph VC}.

\begin{dfn}
For $E\subset \F_q^d$, define the \textbf{orthogonality graph} $\cO(E)$ to be the graph with vertex set $E$, where points $x,y\in E$ are adjacent if and only if $x\cdot y=0$.
\end{dfn}

\begin{lem}
\label{degenerate dot vc}
For any $E\subset \F_q^d\setminus S_0^{d-1}$, the VC-dimension of $\cO(E)$ is at most $d-1$.
\end{lem}

\begin{proof}
Suppose for contradiction the points $\{x_1,\dots,x_d\}\subset E$ and $\{y_I:I\subset [d]\}\subset E$ are such that $x_i\cdot y_I=0$ if and only if $i\in I$.  There are two cases:

\textbf{Case I:} $\{x_1,\dots,x_d\}$ is linearly independent.  Then, there exist $a_1,\dots,a_d\in\F_q$ such that
\[
y_{[d]}=\sum_{i=1}^d a_ix_i.
\]
Therefore,
\[
\|y_{[d]}\|=\sum_{i=1}^d a_ix_i\cdot y_{[d]}=0.
\]
This contradicts the assumption that $E\cap S_0^{d-1}=\emp$.

\textbf{Case II:} $\{x_1,\dots,x_d\}$ is linearly dependent.  In this case, one of the points is in the span of the others; without loss of generality, suppose there exist $b_1,\dots,b_{d-1}$ such that
\[
x_d=\sum_{i=1}^{d-1} b_ix_i.
\]
It follows that
\[
x_d\cdot y_{[d-1]}=\sum_{i=1}^{d-1}b_ix_i\cdot y_{[d-1]}=0.
\]
This contradicts the assumption that $x_i\cdot y_I=0$ if and only if $i\in I$.
\end{proof}

\begin{thm}
\label{sharpness of VC threshold}
Let $d\geq 2$.  For any $t\in\F_q\setminus\{0\}$ there exists a set $E\subset \F_q^d$ (depending on $t$) with $|E|\approx q^{d-1}$ such that the VC-dimension of $G_t(E)$ is at most $d-1$.
\end{thm}
\begin{proof}
Let $E=S_{t/2}^{d-1}$.  For any $x,y\in E$, we have
\[
\|x-y\|=(x-y)\cdot (x-y)=\|x\|-2x\cdot y+\|y\|=t-2x\cdot y.
\]
In particular, $\|x-y\|=t$ if and only if $x\cdot y=0$, hence $G_{t}(E)\simeq \cO(E)$.  Since we have assumed $t\neq 0$, we have $E\cap S_0^{d-1}=\emp$.  The result then follows from Lemma \ref{degenerate dot vc}.   
\end{proof}

\bibliographystyle{plain}
\bibliography{refsVCsupersaturation}

@article {HMS26,
    AUTHOR = {Housholder, Christopher and Mangiapanello, Layna and Senger,
              Steven},
     TITLE = {V{C}-dimension of subsets of {H}amming graphs},
   JOURNAL = {Graphs Combin.},
  FJOURNAL = {Graphs and Combinatorics},
    VOLUME = {42},
      YEAR = {2026},
    NUMBER = {3},
     PAGES = {Paper No. 52, 24},
      ISSN = {0911-0119,1435-5914},
   MRCLASS = {52C45},
  MRNUMBER = {5074523},
       DOI = {10.1007/s00373-026-03046-4},
       URL = {https://doi.org/10.1007/s00373-026-03046-4},
}

@misc{DM25,
      title={{VC}-dimension of Salem sets over finite fields}, 
      author={Moustapha Diallo and Brian McDonald},
      year={2025},
      eprint={2511.08963},
      archivePrefix={arXiv},
      primaryClass={math.CO},
      url={https://arxiv.org/abs/2511.08963}, 
      note={https://arxiv.org/abs/2511.08963}
}

@article{VCdot,
	author = {Iosevich, A. and McDonald, B. and Sun, M.},
	doi = {10.1016/j.disc.2022.113096},
	fjournal = {Discrete Mathematics},
	issn = {0012-365X},
	journal = {Discrete Math.},
	mrclass = {68Q32 (42B10)},
	mrnumber = {4475952},
	number = {1},
	pages = {Paper No. 113096, 9},
	title = {Dot products in {$\Bbb F_q^3$} and the {V}apnik-{C}hervonenkis dimension},
	url = {https://doi.org/10.1016/j.disc.2022.113096},
	volume = {346},
	year = {2023}}

@article {RS26,
    AUTHOR = {Rodgers, Brad and Sahay, Anurag},
     TITLE = {The {VC}-{D}imension of {R}andom {S}ubsets of {F}inite
              {G}roups},
   JOURNAL = {Random Structures Algorithms},
  FJOURNAL = {Random Structures \& Algorithms},
    VOLUME = {69},
      YEAR = {2026},
    NUMBER = {1},
     PAGES = {Paper No. e70089},
      ISSN = {1042-9832,1098-2418},
   MRCLASS = {60C05 (05C25 11B75)},
  MRNUMBER = {5120134},
       DOI = {10.1002/rsa.70089},
       URL = {https://doi.org/10.1002/rsa.70089},
}

@article {MSW25,
    AUTHOR = {McDonald, Brian and Sahay, Anurag and Wyman, Emmett L.},
     TITLE = {The {VC} dimension of quadratic residues in finite fields},
   JOURNAL = {Discrete Math.},
  FJOURNAL = {Discrete Mathematics},
    VOLUME = {348},
      YEAR = {2025},
    NUMBER = {1},
     PAGES = {Paper No. 114192, 13},
      ISSN = {0012-365X,1872-681X},
   MRCLASS = {05C69 (11T24)},
  MRNUMBER = {4787319},
       DOI = {10.1016/j.disc.2024.114192},
       URL = {https://doi.org/10.1016/j.disc.2024.114192},
}

@article {Smalldot,
    AUTHOR = {Ascoli, Ruben and Betti, Livia and Cheigh, Justin and
              Iosevich, Alex and Jeong, Ryan and Liu, Xuyan and McDonald,
              Brian and Milgrim, Wyatt and Miller, Steven J. and Romero
              Acosta, Francisco and Velazquez Innuzzelli, Santiago},
     TITLE = {V{C}-dimension of hyperplanes over finite fields},
   JOURNAL = {Graphs Combin.},
  FJOURNAL = {Graphs and Combinatorics},
    VOLUME = {41},
      YEAR = {2025},
    NUMBER = {2},
     PAGES = {Paper No. 47, 13},
      ISSN = {0911-0119,1435-5914},
   MRCLASS = {12E20 (68Q32)},
  MRNUMBER = {4879034},
MRREVIEWER = {Fengwei\ Li},
       DOI = {10.1007/s00373-025-02909-6},
       URL = {https://doi.org/10.1007/s00373-025-02909-6},
}

@article {Smalldist,
    AUTHOR = {Ascoli, Ruben and Betti, Livia and Cheigh, Justin and
              Iosevich, Alex and Jeong, Ryan and Liu, Xuyan and McDonald,
              Brian and Milgrim, Wyatt and Miller, Steven J. and Romero
              Acosta, Francisco and Velazquez Iannuzzelli, Santiago},
     TITLE = {V{C}-dimension and distance chains in {$\Bbb F^d_q$}},
   JOURNAL = {Korean J. Math.},
  FJOURNAL = {The Korean Journal of Mathematics},
    VOLUME = {32},
      YEAR = {2024},
    NUMBER = {1},
     PAGES = {43--57},
      ISSN = {1976-8605,2288-1433},
   MRCLASS = {68Q32},
  MRNUMBER = {4736053},
MRREVIEWER = {Hans-Ulrich\ Simon},
       DOI = {10.11568/kjm.2024.32.1.43},
       URL = {https://doi.org/10.11568/kjm.2024.32.1.43},
}

@article{BHIPR,
	author = {Bennett, Michael and Hart, Derrick and Iosevich, Alex and Pakianathan, Jonathan and Rudnev, Misha},
	doi = {10.1515/forum-2015-0251},
	fjournal = {Forum Mathematicum},
	issn = {0933-7741},
	journal = {Forum Math.},
	mrclass = {52C10 (42B10)},
	mrnumber = {3592595},
	mrreviewer = {Serge\u{\i} V. Konyagin},
	number = {1},
	pages = {91--110},
	title = {Group actions and geometric combinatorics in {$\Bbb{F}_q^d$}},
	url = {https://doi.org/10.1515/forum-2015-0251},
	volume = {29},
	year = {2017}}

@book{Lang,
	author = {Lang, Serge},
	doi = {10.1007/978-1-4613-0041-0},
	edition = {third},
	isbn = {0-387-95385-X},
	mrclass = {00A05 (15-02)},
	mrnumber = {1878556},
	pages = {xvi+914},
	publisher = {Springer-Verlag, New York},
	series = {Graduate Texts in Mathematics},
	title = {Algebra},
	url = {https://doi.org/10.1007/978-1-4613-0041-0},
	volume = {211},
	year = {2002}}

@Article{FIMW,
  author   = {Fitzpatrick, David and Iosevich, Alex and McDonald, Brian and Wyman, Emmett},
  journal  = {Discrete Comput. Geom.},
  title    = {The {VC}-dimension and point configurations in {$\Bbb F_q^2$}},
  year     = {2024},
  issn     = {0179-5376,1432-0444},
  number   = {4},
  pages    = {1167--1177},
  volume   = {71},
  doi      = {10.1007/s00454-023-00570-5},
  fjournal = {Discrete \& Computational Geometry. An International Journal of Mathematics and Computer Science},
  mrclass  = {68Q32 (52C10)},
  mrnumber = {4742200},
  url      = {https://doi.org/10.1007/s00454-023-00570-5},
}

@Article{BCCHIP16,
  author     = {Bennett, Michael and Chapman, Jeremy and Covert, David and Hart, Derrick and Iosevich, Alex and Pakianathan, Jonathan},
  journal    = {J. Korean Math. Soc.},
  title      = {Long paths in the distance graph over large subsets of vector spaces over finite fields},
  year       = {2016},
  issn       = {0304-9914},
  number     = {1},
  pages      = {115--126},
  volume     = {53},
  doi        = {10.4134/JKMS.2016.53.1.115},
  fjournal   = {Journal of the Korean Mathematical Society},
  mrclass    = {52C10 (05C12 05C90 11T23)},
  mrnumber   = {3450941},
  mrreviewer = {Serge\u{\i} V. Konyagin},
  url        = {https://doi-org.proxy.lib.ohio-state.edu/10.4134/JKMS.2016.53.1.115},
}

@Book{ShalevShwartz2014,
  author    = {Shalev-Shwartz, Shai and Ben-David, Shai},
  publisher = {Cambridge University Press},
  title     = {Understanding Machine Learning - From Theory to Algorithms.},
  year      = {2014},
  isbn      = {978-1-10-705713-5},
  ee        = {http://www.cambridge.org/de/academic/subjects/computer-science/pattern-recognition-and-machine-learning/understanding-machine-learning-theory-algorithms},
  pages     = {I-XVI, 1-397},
}

@Article{IMMM25,
  author        = {Iosevich, Alex and Magyar, Akos and McDonald, Alex and McDonald, Brian},
  journal       = {https://arxiv.org/abs/2510.13984},
  title         = {The {VC}-dimension and point configurations in $\mathbb{R}^d$},
  year          = {2025},
  month         = oct,
  archiveprefix = {arXiv},
  copyright     = {arXiv.org perpetual, non-exclusive license},
  doi           = {10.48550/ARXIV.2510.13984},
  eprint        = {2510.13984},
  primaryclass  = {math.CA},
  publisher     = {arXiv},
}

\end{document}